\documentclass[a4paper,12pt]{article}
\usepackage{graphicx,amsmath,latexsym,amssymb,amsthm}
\newtheorem{thm}{Theorem}[section]

\newtheorem{defn}[thm]{Definition}
\newtheorem{rem}[thm]{Remark}

\newtheorem{exm}[thm]{Example}
\usepackage[colorlinks,urlcolor=blue]{hyperref}

\usepackage{enumerate}
\date{~}
\begin{document}

\title{Boundary and Exterior of a Multiset Topology}
\author{J. Mahanta*, D. Das$^\perp$ \\
Department of Mathematics\\
         NIT Silchar, Assam, 788 010, India.\\
         *juthika.nits@gmail.com, $^\perp$ deboritadas1988@gmail.com}

\maketitle

\begin{abstract}
The concepts of exterior and boundary in multiset topological space are introduced.  We further established few relationships between the concepts of boundary, closure, exterior and interior of an $M$- set. These concepts have been pigeonholed by other existing notions viz., open sets, closed sets, clopen sets and limit points. The necessary and sufficient condition for a multiset to have an empty exterior is also discussed.
\end{abstract}

\vspace{.3 cm}

\section{Introduction}
The theory of sets is indispensable to the world of mathematics. But in set theory 
where repetitions of objects are not allowed it often become difficult to complex systems. If one considers those complex systems where repetitions of objects become certainly inevitable, the set theoretical concepts fails and thus one need more sophisticated tools to handle such situations. This led to the initiation of multiset (M-set) theory by Blizard \cite{BW} in 1989 as a generalization of set theory.
Multiset theory was further studied by Dedekind \cite{DR} by considering each element in the range of a function to have a multiplicity equal to the number of elements in the domain that are mapped to it. The theory of multisets have been studied by many other authors in different senses \cite{GK}, \cite{HM}, \cite{LA}, \cite{SD}, \cite{SDA}, \cite{SS} and \cite{WH}.\\

Since its inception M-set theory have been receiving considerable attention from researchers and wide application of the same can be found in literature [\cite{SDA},\cite{SS}, \cite{YY} etc ]. Algebraic structures for multiset space have been constructed by Ibrahim {\it{et al}}. in \cite{IA}. In \cite{OF}, use of multisets in colorings of graphs have been discussed by F. Okamota {\it{et al}}.  Application of M-set theory in decision making can be seen in \cite{YY}. A. Syropoulos \cite{SA}, presented a categorical approach to multisets along with partially ordered multisets. V. Venkateswaran \cite{VV} found a large new class of multisets Wilf equivalent pairs which is claimed to be the most general multiset Wilf equivalence result to date. In 2012, Girish and John \cite{GJ} introduced multiset topologies induced by multiset relations. The same authors further studied the notions of open sets, closed sets, basis, sub basis, closure and interior, continuity and related properties in M-topological spaces in \cite{GKJ}. Further the concepts of semi open, semi closed multisets were introduced in \cite{JD}, which were then used to study semicompactness in multiset topology.\\

In this paper, we introduce the concept of exterior and boundary in multiset topological space. We begin with preliminary notions
and definitions of multiset theory in section 2. Section 3 which contains main results forms the most fundamental part of the paper and it is followed by section 4 which contains the concluding remarks.

\section{Preliminaries}
Below are some definitions and results as discussed in \cite{GJ}, which are required throughout the paper.

\begin{defn} 
An M-set $M$ drawn from the set $X$ is represented by a function Count $M$ or $C_M : X \longrightarrow W$, where $W$ represents the set of whole numbers.\\

Here $C_M (x)$ is the number of occurrences of the element $x$ in the M-set $M$. We represent the M-set $M$ drawn from the set $X = \{x_1, . . . , x_n\}$ as $M= \{m_1/x_1, m_2/x_2, . . . , m_n/x_n\}$ where $m_i$ is the number of occurrences of the element $x_i, i=1, 2, . . ., n$ in the M-set $M$. Those elements which are not included in the M-set have zero count. 
\end{defn}
Note: Since the count of each element in an $M$-set is always a non-negative integer so we have taken $W$ as the range space instead of $N$.
\begin{exm}
Let $X = \{a, b, c\}$. Then $M = \{3/a, 5/b, 1/c\}$ represents an M-set drawn from $X$.
\end{exm}

Various operations on M-sets are defined as follows:

If $M$ and $N$ are two M-sets drawn from the set $X$, then
\begin{itemize} 
\item $M=N \Leftrightarrow C_M (x) = C_ N (x)~ \forall  x \in X$.
\item $M \subseteq N \Leftrightarrow C_M (x) \leq  C_ N (x) ~ \forall  x \in X$.
\item $P = M \cup N \Leftrightarrow C_P (x) = max \{C_ M (x), C_ N (x)\} ~ \forall x \in X$.
\item $P = M \cap N \Leftrightarrow C_P (x) = min \{C_ M (x), C_ N (x)\} ~ \forall x \in X$.
\item $P = M \oplus N \Leftrightarrow C_P (x) = C_ M (x) + C_ N (x) ~ \forall x \in X$.
\item $P = M \ominus N \Leftrightarrow C_P (x) = max \{C_ M (x)- C_ N (x), 0\} ~ \forall x \in X$, where $\oplus$ and $\ominus$ represents M-set addition and M-set subtraction respectively. 
\end{itemize}
\textbf{Operations under collections of M-sets:} Let $[X]^w$ be an M-space and $\{M_i ~|~i \in I \}$ be a collection of M-sets drawn from $[X]^w$. Then the following operations are defined
\begin{itemize}
\item $\underset{i \in I} {\bigcup} M_i = \{C_{M_i}(x)/x ~ |~ C_{M_i}(x) = max \{C_{M_i}(x) ~|~ x \in X\}\}$.
\item $\underset{i \in I} {\bigcap} M_i = \{C_{M_i}(x)/x ~ |~ C_{M_i}(x) = min \{C_{M_i}(x) ~|~ x \in X\}\}$.
\end{itemize}

\begin{defn}
The support set of an M-set $M$, denoted by $M^*$ is a subset of $X$ and is defined as $M^* = \{x \in X ~|~ C_M(x) > 0\}$. $M^*$ is also called root set.
\end{defn}

\begin{defn}
An M-set $M$ is called an empty M-set if $C_M(x) = 0, ~~ \forall  ~ x \in X$.
\end{defn}

\begin{defn}
A domain $X$, is defined as the set of elements from which M-set are constructed. The M-set space $[X]^w$ is the set of all M-sets whose elements are from X such that no element occurs more than w times.
\end{defn}

\begin{rem}
It is clear that the definition of the operation of M-set addition is not valid in the context of M-set space $[X]^w$, hence it was refined as\\
$C_{M_1 \oplus M_2}(x) = min \{w, C_{M_1}(x) + C_{M_2}(x)\}$ for all $x \in X$.
\end{rem}

 In multisets the number of occurrences of each element is allowed to be more than one which leads to generalization of the definition of subsets in classical set theory. So, in contrast to classical set theory, there are different types of subsets in multiset theory.

\begin{defn}
A subM-set N of M is said to be a whole subM-set if and only if $C_{N}(x)=C_{M}(x)$ for every $x \in N$.
\end{defn}

\begin{defn}
A subM-set N of M is said to be a partial whole subM-set if and only if $C_{N}(x)=C_{M}(x)$ for some $x \in N$.
\end{defn}

\begin{defn}
A subM-set N of M is said to be a full subM-set if and only if $C_{N}(x) \leq C_{M}(x)$ for every $x \in N$.
\end{defn}

As various subset relations exist in multiset theory, the concept of power M-set can also be generalized as follows:
\begin{defn} Let $M \in [X]^w$ be an M-set.
\begin{itemize}
\item The power M-set of M denoted by ${\mathcal{P}}(M)$ is defined as the set of all subM-sets of M.
\item The power whole M-set of M denoted by ${\mathcal{PW}}(M)$ is defined as the set of all whole subM-sets of M.
\item The power full M-set of M denoted by ${\mathcal{PF}}(M)$ is defined as the set of all full subM-sets of M.
\end{itemize}
\end{defn}
The power set of an M-set is the support set of the power M-set and is denoted by ${\mathcal{P}}^{*}(M)$.
\begin{defn}
Let $M \in [X]^w$ and $\tau \subseteq {\mathcal{P}}^{*}(M) $. Then $\tau$ is called an M-topology if it satisfies the following properties:
\begin{itemize}
\item The M-set M and the empty M-set $\phi$ are in $\tau$.
\item The M-set union of the elements of any subcollection of $\tau$ is in $\tau$.
\item The M-set intersection of the elements of any finite subcollection of $\tau$ is in $\tau$.
\end{itemize}
The elements of $\tau $ are called open M-set and their complements are called closed M-sets.
\end{defn}

\begin{defn}
Given a subM-set A of an M-topological space M in $[X]^w$
\begin{itemize}
\item The interior of A is defined as the M-set union of all open M-sets contained in A and is denoted by $int(A)$ i.e.,
$C_{int(A)}(x)= C_{\cup G}(x)$ where G is an open M-set and $G \subseteq A$.
\item The closure of A is defined as the M-set intersection of all closed M-sets containing A and is denoted by $cl(A)$ i.e.,
$C_{cl(A)}(x)= C_{\cap K}(x)$ where G is a closed M-set and $A \subseteq K$.
\end{itemize}
\end{defn}

\begin{defn} If $M$ is an M-set, then the M-basis for an M-topology in $[X]^w$ is a collection $\mathcal{B}$ of subM-sets of $M$ such that 
\begin{itemize}
\item For each $x {\in}^m M$, for some $m>0$ there is at least one M-basis element $B \in \mathcal{B}$ containing $m/x$.
\item If $m/x$ belongs to the intersection of two M -basis elements $P$ and $Q$, then $\exists$ an M- basis element $R$ containing $m/x$ such that $R \subseteq P \cap Q$ with $C_R(x)= C_{P \cap Q}(x)$ and $C_R(y) \leq  C_{P \cap Q}(y)$ $\forall y \neq x$.
\end{itemize}
\end{defn}

\begin{defn}
Let $(M, \tau)$ be an M-topological space in $[X]^w$ and A is a subM-set of M. If $k/x$ is an element of M then $k/x$ is a limt point of an M-set when every neighborhood of $k/x$ intersects A in some point (point with non-zero multiplicity) other than $k/x$ itself.
\end{defn}

\begin{defn}
Let $(M, \tau)$ be an M-topological space and N is a subM-set of M. The collection $\tau_N = \{N \cap U : U \in \tau\}$ is an M-topology on N, called the subspace M-topology. With this M-topology, N is called a subspace of M.
\end{defn}

Throughout the paper we shall follow the following definition of complement in an M- topological space.
\begin{defn} \cite{JD}
The M-complement of a subM-set $N$ in an M-topological space $(M, \tau)$ is denoted and defined as $N^{c}= M \ominus N$.
\end{defn}

\section{Exterior and Boundary of Multisets}

The notions of interior and closure of an M-set in M-topology have been introduced and studied by Jacob et al. \cite{GJ}. The other topological structures like exterior and boundary have remain untouched. In this section, we introduce the concepts of exterior and boundary in multiset topology.
\\ Consider an M-topological space $(M, \tau)$ in $[X]^w$.

\begin{defn}
The exterior of an M-set A in M is defined as the interior of M-complement of A and is denoted by ext(A), i.e., 
\begin{center}
 {\Large{C}}$_{ext(A)}(x) $= {\Large{C}}$_{int (A^c) }(x)$ for all $x \in X$.
\end{center}
\end{defn}

\begin{exm}
Let $X= \lbrace a,b \rbrace$,$w=3$ and $M= \lbrace 2/a, 3/b \rbrace$. We consider the topology $\tau=  \lbrace \phi, M, \lbrace 1/a\rbrace, \lbrace 2/b \rbrace, \lbrace 1/a, 2/b \rbrace\rbrace$ on $M$. Then exterior of the M-set $A=\lbrace 1/a, 3/b \rbrace$ is $\lbrace 1/a \rbrace$.
\end{exm}

\begin{rem}
$ext(A)$ is the largest open subM-set contained in $A^c$.
\end{rem}

\begin{defn}
The boundary of an M-set A is the M-set of elements which does not belong to the interior or the exterior of A. In other words, the boundary of an M-set A is the M-set of elements which belongs to the intersection of closure of A and closure of M-complement of A. It is denoted by bd(A).
\begin{center}
 {\Large{C}}$_{bd(A)}(x)$ = {\Large{C}}$_{cl(A) \cap cl(A^c)}(x)$ for all $x \in X$.
\end{center}
\end{defn}

\begin{exm}
Let $X=\lbrace a,b,c,d \rbrace$, $w=5$ and $M= \lbrace 5/a, 3/b, 5/c, 5/d \rbrace$. We consider the topology  $\tau= \lbrace \phi, M, \lbrace 1/a, 2/b, 3/c, 2/d \rbrace, \lbrace 1/a, 3/c \rbrace, \lbrace 2/b, 5/d \rbrace,$\\$ \lbrace 1/a, 2/b, 3/c, 5/d \rbrace, \lbrace 2/b, 2/d \rbrace \rbrace$ on $M$.Then  for any set $A=\lbrace 3/a, 3/b, 3/c, 3/d \rbrace$ we have $cl(A)= M$ and $cl(A^{c})= \lbrace 4/a, 1/b, 2/c, 3/d \rbrace$. Hence, $bd(A)= \lbrace 4/a, 1/b, 2/c, 3/d \rbrace$.
\end{exm}

\begin{rem}
 $bd(A)$ is the smallest closed subM-set containing $A^c$.
\end{rem}
\begin{rem}
$A$ and $A^c$ both have same boundary. 
\end{rem}

\begin{thm}
Let $(M, \tau)$ be an M-topological space. Then 
\begin{enumerate}[(i)]
\item {\Large{C}}$_{ext(A \cup B)}(x)= $ {\Large{C}}$_{ext(A) \cap ext(B)}(x)$ $,\forall x \in X$.
\item {\Large{C}}$_{ext(A \cap B)}(x) \geq $ {\Large{C}}$_{ext(A) \cup ext(B)}(x)$ $,\forall x \in X$.
\end{enumerate}
\end{thm}

\begin{proof}
\begin{enumerate}[(i)]
\item From the definition of exterior,
 \begin{eqnarray*}
{\Large{C}}_{ext(A \cup B)}(x) &=&  {\Large{C}}_{int(A \cup B)^c}(x) \\
&=& {\Large{C}}_{int(A^c \cap B^c)}(x)\\
&=& {\Large{C}}_{int(A^c)\cap int(B^c)}(x)\\
&=& {\Large{C}}_{ext(A) \cap ext(B)}(x), \forall x \in X.
\end{eqnarray*}
\item \begin{eqnarray*}
{\Large{C}}_{ext(A \cap B)}(x) &=&  {\Large{C}}_{int(A \cap B)^c}(x) \\
&=& {\Large{C}}_{int(A^c \cup B^c)}(x)\\
&\geq & {\Large{C}}_{int(A^c)\cup int(B^c)}(x)\\
&=& {\Large{C}}_{ext(A) \cup ext(B)}(x), \forall x \in X.
\end{eqnarray*}
\end{enumerate}
\end{proof}

\begin{thm}
Let $(M, \tau)$ be an M-topological space in $[X]^w$. For any two M-sets $A$ and $B$ in $M$, the following results hold:
\begin{enumerate}[(i)]
\item {\Large{C}}$_{(bd(A))^c}(x)$ = {\Large{C}}$_{int(A) \cup int(A^c)}(x) = $ {\Large{C}}$_{int(A) \cup ext(A)}(x)$  
\item {\Large{C}}$_{cl(A)}(x)$ = {\Large{C}}$_{int(A) \cup bd(A)}(x)$
\item {\Large{C}}$_{bd(A)}(x)$ = {\Large{C}}$_{cl(A) \ominus int(A)}(x) $
\item {\Large{C}}$_{int(A)}(x)$ = {\Large{C}}$_{A \ominus bd(A)}(x) $
\end{enumerate}
\end{thm}

\begin{proof}
 \begin{eqnarray*} (i) {\Large{C}}_{(bd(A))^c}(x) &=& {\Large{C}}_{(cl(A) \cap cl(A^c))^c}(x)\\
&=& {\Large{C}}_{(cl(A))^c \cup (cl(A^c))^c}(x)\\
&=& {\Large{C}}_{int(A^c) \cup int(A)}(x)\\
&=& {\Large{C}}_{int(A) \cup ext(A)}(x).
\end{eqnarray*}
\begin{eqnarray*} (ii)  {\Large{C}}_{int(A) \cup bd(A)}(x) &=&  {\Large{C}}_{int(A) \cup (cl(A) \cap cl(A^c))}(x)\\
 &=&  {\Large{C}}_{(int(A) \cup (cl(A)) \cap (int(A) \cup cl(A^c))}(x)\\
  &=&  {\Large{C}}_{ cl(A) \cap (int(A) \cup (int(A))^c)}(x)\\
   &=&  {\Large{C}}_{ cl(A)}(x).
\end{eqnarray*}
(iii) We have  {\Large{\it C}}$_{cl(A) \ominus int(A)}(x) = max \{$ {\Large{\it C}}$_{cl(A)}(x) - ${\Large{\it C}}$_{int(A)}(x), 0 \}$ 

\begin{itemize}
\item Case 1: max is 0\\
So we must have {\Large{\it C}}$_{cl(A)}(x) =$ {\Large{\it C}}$_{int(A)}(x)$.\\
Then  {\Large{\it C}}$_{bd(A)}(x) =$ {\Large{\it C}}$_{int(A) \cap cl(A^c)}(x)=$ {\Large{\it C}}$_{int(A) \cap (int(A))^c}(x)= $ {\Large{\it C}}$_{\phi}(x)$.

\item Case 2: max is {\Large{\it C}}$_{cl(A)}(x) - ${\Large{\it C}}$_{int(A)}(x)$.
Then \\
\begin{eqnarray*}{\Large{\it C}}_{bd(A)}(x) &=& {\Large{\it C}}_{cl(A) \cap cl(A^c)}(x) \\ &=& {\Large{\it C}}_{cl(A) \cap (int(A))^c}(x) \\ &=&  {\Large{\it C}}_{cl(A)}(x) - {\Large{\it C}}_{int(A)}(x). \end{eqnarray*}
\end{itemize}
(iv) We have  {\Large{\it C}}$_{A \ominus bd(A)}(x) = max \{$ {\Large{\it C}}$_{A}(x) - ${\Large{\it C}}$_{bd(A)}(x), 0 \}$ 
\begin{itemize}
\item Case 1: When max is 0 proof is trivial.
\item Case 2: When max is {\Large{\it C}}$_{A}(x) - ${\Large{\it C}}$_{bd(A)}(x)$, we have\\
\begin{eqnarray*} 
{\Large{\it C}}_{A \ominus bd(A)}(x) &=& {\Large{\it C}}_{A \cap (bd(A))^c}(x)\\
&=& {\Large{\it C}}_{A \cap (int(A) \cup ext(A))}(x)\\
&=& {\Large{\it C}}_{A \cap (int(A) \cup int(A^c))}(x)\\
&=& {\Large{\it C}}_{(A \cap int(A)) \cup ( A \cap int(A^c))}(x)\\
&=& {\Large{\it C}}_{ int(A) \cup \phi}(x)\\
&=& {\Large{\it C}}_{ int(A)}(x).
\end{eqnarray*}
\end{itemize}
\end{proof}

The following three theorems characterize the open and closed M-sets in terms of boundary.
\begin{thm}
Let $A$ be a subM-sets in an M-topology $(M, \tau)$. Then $A$ is open if and only if {\Large{C}}$_{A \cap bd(A)}(x)= 0 $, $\forall x \in X.$
\end{thm}

\begin{proof}
 Let $A$ be an open M-set. Then {\Large{\it C}}$_{int(A)}(x)=$ {\Large{\it C}}$_{A}(x)$ , $\forall x \in X$. Now, {\Large{\it C}}$_{A \cap bd(A)}(x)= $ {\Large{\it C}}$_{int(A) \cap bd(A)}(x)= 0$. 
 \\Conversely, let $A$ be an M-set such that {\Large{\it C}}$_{A \cap bd(A)}(x) = 0 \Rightarrow $ {\Large{\it C}}$_{A \cap (cl(A)\cap cl(A^c))}(x) = 0 \Rightarrow $ {\Large{\it C}}$_{A \cap cl(A^c)}(x) = 0 \Rightarrow $ {\Large{\it C}}$_{ cl(A^c)}(x) \leq $ {\Large{\it C}}$_{A^c}(x) \Rightarrow  A^c$ is closed M-set $\Rightarrow A$ is open M-set.
\end{proof}

\begin{thm}
Let $A$ be a subM-sets in an M-topology $(M, \tau)$. Then$A$ is closed if and only if {\Large{C}}$_{bd(A)}(x) \leq $ {\Large{C}}$_A(x)$ , $\forall x \in X.$
\end{thm}
\begin{proof}
 Let $A$ be a closed M-set. Then {\Large{\it C}}$_{cl(A)}(x)=$ {\Large{\it C}}$_{A}(x)$ , $\forall x \in X$. Now, {\Large{\it C}}$_{bd(A)}(x)= $ {\Large{\it C}}$_{cl(A) \cap cl(A^c)}(x) \leq$ {\Large{\it C}}$_{cl(A)}(x)=$ {\Large{\it C}}$_A(x)$, $\forall x \in X$. 
 \\Conversely,let {\Large{\it C}}$_{bd(A)}(x) \leq $ {\Large{\it C}}$_A(x) \Rightarrow$ {\Large{\it C}}$_{bd(A) \cap A^c}(x) =0 \Rightarrow$ {\Large{\it C}}$_{bd(A^c) \cap A^c}(x) =0 $. Therefore, $A^c$ is an open M-set. Hence, A is a closed M-set.
\end{proof}

\begin{thm}
Let $A$ be a subM-sets in an M-topology $(M, \tau)$. Then $A$ is clopen if and only if {\Large{C}}$_{bd(A)}(x) =0 .$ 
\end{thm}

\begin{proof}
Let {\Large{\it C}}$_{bd(A)}(x) =0 \Rightarrow$ {\Large{\it C}}$_{cl(A) \cap cl(A^c)}(x) =0 \Rightarrow$ {\Large{\it C}}$_{cl(A)}(x) \leq$ {\Large{\it C}}$_{(cl(A^c))^c}(x) \Rightarrow $ {\Large{\it C}}$_{cl(A)}(x) \leq$ {\Large{\it C}}$_{int(A)}(x) \leq$  {\Large{\it C}}$_A(x)\Rightarrow A$ is a closed M-set.
\\Again, {\Large{\it C}}$_{cl(A) \cap cl(A^c)}(x) =0 \Rightarrow$ {\Large{\it C}}$_{cl(A)\cap (int(A))^c}(x) \Rightarrow$ {\Large{\it C}}$_{A \cap (int(A))^c}(x) \Rightarrow $ {\Large{\it C}}$_A(x) \leq $ {\Large{\it C}}$_{int(A)}(x) \Rightarrow A$ is an open M-set.
\\Conversely, let $A$ be both open and closed M-set.
\\ Then {\Large{\it C}}$_{bd(A)}(x) =$ {\Large{\it C}}$_{cl(A) \cap cl(A^c)}(x)=$ {\Large{\it C}}$_{cl(A) \cap (int(A))^c}(x) =$ {\Large{\it C}}$_{A \cap A^c}(x) =0$.   
\end{proof}

\begin{thm}
For any two M-sets $A$ and $B$ in $(M, \tau)$ the followings hold true:
\begin{enumerate}[(i)]
\item {\Large{C}}$_{bd(A \cup B)}(x)  \leq $  {\Large{C}}$_{bd(A) \cup bd(B)}(x)$, $\forall x \in X.$
\begin{proof} Let $A$ and $B$ be any two M-sets in $(M, \tau)$. Then,
 \begin{eqnarray*} {\Large{C}}_{bd(A \cup B)}(x)  &=&  {\Large{C}}_{cl(A \cup B)  \cap cl(A \cup B)^c}(x)\\
  &=&  {\Large{C}}_{[cl(A) \cup cl(B)]  \cap [cl(A^c) \cap cl(B^c)]}(x)\\
   &=&  {\Large{C}}_{[(cl(A) \cap cl(A^c))  \cap (cl(A) \cap cl(B^c))] \cup [(cl(B) \cap cl(A^c))\cap (cl(B) \cap cl(B^c))]}(x)\\
   &=&  {\Large{C}}_{[bd(A)  \cap (cl(A) \cap cl(B^c))] \cup [(cl(B) \cap cl(A^c))\cap bd(B)]}(x)\\
   &\leq &  {\Large{C}}_{bd(A)  \cup  bd(B)}(x).   
\end{eqnarray*} 
\end{proof}
\item  {\Large{C}}$_{bd(A \cap B)}(x)  \leq $  {\Large{C}}$_{bd(A) \cap bd(B)}(x)$, $\forall x \in X.$
\begin{proof} Let $A$ and $B$ be any two M-sets in $(M, \tau)$. Then,
\begin{eqnarray*} {\Large{C}}_{bd(A \cap B)}(x)  &=&  {\Large{C}}_{cl(A \cap B)  \cap cl(A \cap B)^c}(x)\\
  &=&  {\Large{C}}_{[cl(A) \cap cl(B)]  \cap [cl(A^c) \cup cl(B^c)]}(x)\\
   &=&  {\Large{C}}_{[(cl(A) \cap cl(A^c))  \cap (cl(A) \cap cl(B^c))] \cap [(cl(B) \cap cl(A^c))\cap (cl(B) \cap cl(B^c))]}(x)\\
   &=&  {\Large{C}}_{[bd(A)  \cap (cl(A) \cap cl(B^c))] \cap [(cl(B) \cap cl(A^c))\cap bd(B)]}(x)\\
   &\leq &  {\Large{C}}_{bd(A)  \cap  bd(B)}(x).   
\end{eqnarray*}
\end{proof}
\end{enumerate}
\end{thm}

\begin{thm}
In an M-topological space, for any M-set A $bd(bd(A))$ is a closed M-set.
\end{thm}

\begin{proof}
Let  $bd(A)=B$. Then
\begin{eqnarray*} {\Large{C}}_{cl(bd(bd(A)))}(x) &=& {\Large{C}}_{cl(bd(B))}(x) \\
&=& {\Large{C}}_{cl(cl(B) \cap cl(B^c))}(x) \\
& \leq & {\Large{C}}_{cl(cl(B)) \cap cl(cl(B^c))}(x) \\
&=& {\Large{C}}_{cl(B) \cap cl(B^c)}(x) \\
&=& {\Large{C}}_{bd(B)}(x) \\
&=& {\Large{C}}_{bd(bd(A))}(x). \\
\end{eqnarray*}
i.e., closure of $bd(bd(A))$ is contained in itself and hence is a closed M-set.
\end{proof}

\begin{thm}
In an M-topological space, for any M-set A we have the following:
\begin{enumerate}[(i)]
\item {\Large{C}}$_{bd(bd(A))}(x)  \leq $  {\Large{C}}$_{bd(A)}(x)$, $\forall x \in X.$
\begin{proof} \begin{eqnarray*}{\Large{C}}_{bd(bd(A))}(x)  &=&   {\Large{C}}_{bd(cl(A) \cap cl(A^c))}(x) \\
&=&   {\Large{C}}_{[cl(cl(A) \cap cl(A^c))] \cap [cl(cl(A) \cap cl(A^c))^c]}(x) \\
& \leq &   {\Large{C}}_{[cl(A) \cap cl(A^c)] \cap [cl(int(A^c) \cup int(A))]}(x) \\
&=&   {\Large{C}}_{bd(A) \cap cl(M)}(x) \\
&=&   {\Large{C}}_{bd(A) \cap M}(x) \\
&=&   {\Large{C}}_{bd(A)}(x). \\
\end{eqnarray*}
\end{proof}
\item  {\Large{C}}$_{bd(bd(bd(A)))}(x)  = $  {\Large{C}}$_{bd(bd(A))}(x)$, $\forall x \in X.$
\begin{proof} \begin{eqnarray}{\Large{C}}_{bd(bd(bd(A)))}(x)  &=&   {\Large{C}}_{cl(bd(bd(A))) \cap cl(bd(bd(A)))^c}(x) \\
&=&   {\Large{C}}_{bd(bd(A)) \cap cl(bd(bd(A)))^c}(x).
\end{eqnarray}
  
\begin{eqnarray*}
Now, {\Large{C}}_{(bd(bd(A)))^c}(x) &=& {\Large{C}}_{[cl(bd(A)) \cap cl(bd(A))^c]^c}(x)\\
&=& {\Large{C}}_{[bd(A) \cap cl(bd(A))^c]^c}(x)\\
&=& {\Large{C}}_{(bd(A))^c \cup [cl(bd(A))^c]^c}(x)\\
\end{eqnarray*}
Taking closure on both sides and considering $cl(bd(A))^c= B$, we have
\begin{eqnarray*}
{\Large{C}}_{cl(bd(bd(A)))^c}(x) &=& {\Large{C}}_{B \cup cl(B^c)}(x)\\
& \geq & {\Large{C}}_{B \cup B^c}(x)\\
&=& {\Large{C}}_{M}(x).
\end{eqnarray*}
Now, substituting this in equation(1)
\begin{eqnarray*}{\Large{C}}_{bd(bd(bd(A)))}(x)  &=&   {\Large{C}}_{cl(bd(bd(A))) \cap M}(x) \\
&=& {\Large{C}}_{bd(bd(A))}(x).
\end{eqnarray*}
\end{proof}
\end{enumerate}
\end{thm}

The following theorem decomposes boundary of an M-set.
\begin{thm}
{\Large{C}}$_{bd(A)}(x)= $  {\Large{C}}$_{int(bd(A)) \cup bd(bd(A))}(x)$, $\forall x \in X.$  
\end{thm}

\begin{proof}
 From theorem 3.12(i) and the property of interior i.e.,${\Large{C}}_{int(bd(A))}(x) \leq {\Large{C}}_{bd(A)}(x)$, its obvious that ${\Large{C}}_{int(bd(A)) \cup bd(bd(A))}(x) \leq {\Large{C}}_{bd(A)}(x)$, $\forall x \in X$.
\end{proof}

Following is a theorem to characterize boundary of an M-set in terms of limit points of the set. 
\begin{thm}
An M-set $A$ in an M-topology $(M, \tau)$ contains all its boundary points if and only if it contains all its limit points.
\end{thm}
\label{thm}
\begin{proof}
Suppose $A$ contains all its boundary points and if possible let $k/x$ $\in A^{c}$ be a limit point of A. Since every neighborhood of $k/x$ contains both a point of $A^{c}$ and a point of $A$, we have $k/x$ $\in bd(A)\subseteq cl(A)$, which is a contradiction since $A$ contains all its boundary points.
\\ Conversely, let $A$ contains all its limit points. If $k/x \in A \ominus bd(A)$ and $N$ is a neighborhood of $k/x$ then $N$ contains a point of $A$ which cannot be equal to $k/x$ since $k/x$ $\notin A$. Therefore, $k/x$ is a limit point of $A$ and is not contained in $A$. Hence, $A$ contains all its boundary points.
\end{proof}
 

\begin{thm}
Let $A$ be an M-set in an M-topology $(M, \tau)$. Then $ext(A)$ is empty if and only if every nonempty open M-set in $M$ contains a point of $A$.
\end{thm}

\begin{proof}
Let every non empty open $M$- set in $M, \tau$ contains a point of $A$. Then, every $k/x$ $\in A \subseteq M$ is a limit point if $A$. So, \begin{eqnarray}
 k \leq C_{cl(A)(x)} 
\Rightarrow C_{M(x)} \leq C_{cl(A)(x)}
\end{eqnarray} 
\label{eq3}

\begin{eqnarray*}Now,~ to 	~show ~that
~C_{ext(A)}(x) &=& C_{\phi}(x) \\
\Leftrightarrow C_{int(A^{c})}(x) &=& C_{\phi}(x)\\
\Leftrightarrow C_{(cl(A))^{c}}(x) &=& C_{\phi}(x)\\
\Leftrightarrow C_{cl(A)}(x) &=& C_{M}(x)
\end{eqnarray*} 

But then we have, 
\begin{eqnarray}
C_{cl(A)}(x) \leq C_{M}(x), \forall x.
\end{eqnarray}
\label{eq4}
So, (\ref{eq3}) and (4) imply that $ext(A)$ is empty.

Conversely, let $C_{ext(A)}(x) = C_{\phi}(x)$. Let $O$ be any open $M$ -set in $(M, \tau)$. To show that $O$ contains a point of $A$.\\
Let $k/x$ $\in O$. Since $ext(A)$ is empty so no neighborhood of $k/x$ is contained in $A^{c}$, i.e., all neighborhoods of $k/x$ are contained in $A$. Therefore we have, $C_{O \cap A}(x) \neq C_{\phi}(x)$.
\end{proof}

\section{Conclusion}
The notions of exterior and boundary in context of multiset theory have been introduced and studied in this paper.Some properties of the introduced notions are studied along with their characterization and decomposition. Further, boundary is characterized in terms of open sets, closed sets, clopen sets. Theorem.3.17 characterizes boundary in terms of limit points. The necessary and sufficient condition for an $M$- set to have empty exterior is contemplated by Theorem.3.18. 

Topological and topology-based data are useful for detecting and correcting digitizing errors which occurs in spatial analysis. Keeping this in view, applications of the initiated concepts in those models which are designed using multiset theory can be considered for future work.


\end{document}